\documentclass[11pt,reqno]{amsart}

\usepackage[foot]{amsaddr}
\usepackage{amsmath,amssymb,amsthm}
\usepackage[colorlinks=true,citecolor=black,linkcolor=black,urlcolor=blue]{hyperref}
\usepackage{booktabs}
\usepackage{cellspace}
\usepackage{graphicx}
\usepackage[usenames,dvipsnames]{color}
\usepackage{enumerate}
\usepackage{multicol}
\usepackage{tikz}

\theoremstyle{plain}
\newtheorem{lemma}{Lemma}
\newtheorem{theorem}{Theorem}

\newtheorem{corollary}{Corollary}

\theoremstyle{definition}

\theoremstyle{remark}

\newcommand{\nats}{\mathbb{N}}

\let\leq\leqslant
\let\geq\geqslant
\let\le\leqslant
\let\ge\geqslant

\def\cp{\,\square\,}
\DeclareMathOperator {\dil}{dil}
\DeclareMathOperator {\dist}{dist}
\DeclareMathOperator {\diam}{diam}

\title[Lower bounds for embeddings into hypercubes]{Lower bounds for dilation, wirelength, and edge congestion of embedding graphs into hypercubes}

\author[R.S. Rajan]{R. Sundara Rajan$^{1}$}
\address{$^1$Department of Mathematics, Hindustan Institute of Technology and Science,\\
Padur, Chennai 603103, India}
\author[T. Kalinowski]{Thomas Kalinowski$^{2}$}
\address{$^2$University of New England, School of Science and Technology, Armidale, NSW 2351, Australia}
\author[S. Klav\v zar]{Sandi Klav\v zar$^{3,4,5}$}
\address{$^3$Faculty of Mathematics and Physics, University of Ljubljana, Slovenia}
\address{$^4$Faculty of Natural Sciences and Mathematics, University of Maribor, Slovenia}
\address{$^5$Institute of Mathematics, Physics and Mechanics, Ljubljana, Slovenia}
\author[H. Mokhtar]{Hamid Mokhtar$^{6}$}
\address{$^{6}$School of Information Technology and Electrical Engineering\\
Faculty of Engineering, Architecture and Information Technology\\
St Lucia QLD 4072, Australia}
\author[T.M. Rajalaxmi]{T.M. Rajalaxmi$^{7}$}
\address{$^{7}$Department of Mathematics, SSN College of Engineering, Chennai 603110, India}

\email{vprsundar@gmail.com}
\email{tkalinow@une.edu.au}
\email{sandi.klavzar@fmf.uni-lj.si}
\email{h.mokhtar@uq.edu.au}
\email{laxmi.raji18@gmail.com}

\date{}


\begin{document}

\begin{abstract}
  Interconnection networks provide an effective mechanism for exchanging data between processors in
  a parallel computing system. One of the most efficient interconnection networks is the hypercube
  due to its structural regularity, potential for parallel computation of various algorithms, and
  the high degree of fault tolerance. Thus it becomes the first choice of topological structure of
  parallel processing and computing systems. In this paper, lower bounds for the dilation,
  wirelength, and edge congestion of an embedding of a graph into a hypercube are proved. Two of
  these bounds are expressed in terms of the bisection width. Applying these results, the dilation
  and wirelength of embedding of certain complete multipartite graphs, folded hypercubes, wheels,
  and specific Cartesian products are computed.
\end{abstract}

\maketitle

\section{Introduction}\label{sec:intro}

A suitable interconnection network is an important part for the design of a multicomputer or
multiprocessor system. Such a network is usually modeled by a symmetric graph, where the nodes
represent the processing elements and the edges represent the communication channels. Desirable
properties of an interconnection network include symmetry, embedding capabilities, relatively small
degree, small diameter, scalability, robustness, and efficient routing~\cite{Ramanathan1988}. One of
the most efficient interconnection networks is the hypercube due to its structural regularity,
potential for parallel computation of various algorithms, and the high degree of fault
tolerance~\cite{Saad1988}.

The hypercube has many excellent features and thus becomes the first choice of topological structure
of parallel processing and computing systems. The machines based on hypercubes such as the Cosmic
Cube from Caltech, the iPSC/2 from Intel, and Connection Machines have been implemented
commercially~\cite{Choudum2002}. Hypercubes are very popular models for parallel computation because
of their symmetry and relatively small number of interprocessor connections. Among the
  hypercube-based interconnection networks designed for extremely large-scale supercomputers that
  were studied in the last decade we mention metacubes~\cite{li-2010}, hierarchical cubic
  networks~\cite{bossard-2014, liu-2019}, and $k$-ary $n$-cubes~\cite{Fan2019a, yun-2016}. The
hypercube embedding problem is the problem of mapping a communication graph into a hypercube
multicomputer. Hypercubes are known to simulate other structures such as grids and binary
trees~\cite{Chen1995,Manuel2009}.

Graph embedding is an important technique that maps a guest graph into a host graph, usually an
interconnection network~\cite{Chaudhary1993}. The quality of an embedding can be measured by certain
cost criteria. One of these criteria which is considered very often is the \emph{dilation}. The
dilation of an embedding is defined as the maximum distance between a pair of vertices of the host
graph that are images of adjacent vertices of the guest graph. It is a measure for the communication
time needed when simulating one network on another~\cite{Leighton2014}.  A closely related and
important cost criterion is the {\em wirelength of an embedding} which is the sum of the dilations
in the host graph of edges in the guest graph. The wirelength of a graph embedding arises from VLSI
design, data structures and data representations, networks for parallel computer systems, biological
models that deal with cloning and visual stimuli, parallel architecture, structural engineering and
so on~\cite{Kuo2017, Rajantoappear, Xu2013}.

Another important cost criterion is the \emph{edge congestion}. The edge congestion of an embedding
is the maximum number of edges of the guest graph that are embedded on any single edge of the host
graph. An embedding with a large edge congestion faces many problems, such as long communication
delay, circuit switching, and the existence of different types of uncontrolled noise. In data
networking, network edge congestion occurs when a link or node is carrying so much data that its
quality of service deteriorates. Typical effects include packet loss or the blocking of new
connections. Therefore, a minimum edge congestion is a most desirable feature in network
embedding~\cite{Dvorak2007,Matsubayashi2015}.

Graph embeddings have been well-studied for a number of networks~\cite{Bezrukov2000, Bezrukov1998,
  Chen1995, Dvorak2007, J2007, Fan2006, Fan2019, Fan2019a, Fang2005, Han2010, Lai1999, Manuel2009,
  MANUEL2012, Miller2014, Rajasingh2012, Rajasingh2012a, Shalini2018, Xu2013}. While in actual
supercomputing applications it is usually not sufficient to be able to embed individual graphs
optimally, a good understanding of the embedding behaviour of single graphs can provide insights to
support the development of methodologies and algorithms for solving the more complex problems
arising in practice. In this paper, we study the dilation, wirelength, and congestion of embeddings
of graphs into hypercubes and proceed as follows. In the next section concepts needed are formally
introduced. In Section~\ref{sec:dilation}, we consider the dilation of embedding into $Q_n$ and in
the main result characterize the graphs $G$ with $\dil(G, Q_n)\le n-1$: they are precisely the
graphs which contain a perfect anti-matching. Then, in Section~\ref{sec:two-bounds}, we prove a
general lower bound on the edge congestion and a lower bound on the wirelength of embeddings into
hypercubes. Both bounds are expressed in terms of the bisection width. In the remainder of the paper
these results are applied to particular graphs which have been studied in the supercomputing
literature before~\cite{Lu2009,Guo2013,Bossard2014}. The specific results obtained in
Sections~\ref{sec:complete-multipartite} to \ref{sec:cartesian} are summarized in
Table~\ref{table1}, where $K_{2^{n-p},\dots,2^{n-p}}$, $FQ_n$ and $W_{2^n}$ denote the complete
$2^p$-partite graph with parts of size $2^{n-p}$, the \emph{folded $n$-hypercubes} and the
\emph{$2^n$-wheel}, respectively, and $G\square H$ is the cartesian product of two graphs $G$ and
$H$ (following standard notation in the graph theory literature, see for
instance~\cite{Imrich2008}). All of these graphs are defined precisely in the sections containing
the corresponding results.

\begin {table}[htb]
\caption{Summary of the results in Sections~\ref{sec:complete-multipartite} to
  \ref{sec:cartesian}.}
\label{table1}
\begin{center}
\begin{tabular}{ScScSc}\toprule
\vspace{0.1 cm}
\textbf{\begin{tabular}{c}
          Embedding \\
          $f:G \to Q_n$ \\
        \end{tabular}
} & \textbf{Dilation} & \textbf{Wirelength}  \\ \toprule
    \vspace{0.15 cm}
\vspace{0.15 cm}
   $G=K_{2^{n-p},\dots,2^{n-p}}$  &
                                    \begin{tabular}{lcl}
                                      $n$ &for& $p=n$\\
                                      $n-1$ &for& $n-\log_2(n+1)$\\
                                          &&$<p\leq n-1$\\
                                      \multicolumn{3}{c}{(Corollary~\ref{cor:dilation_n-1})}
                                    \end{tabular}
&
   \begin{tabular}{r}
     $n2^{2n-p-2}(2^{p}-1)$ \\
     (Theorem~\ref{thm:embedding_multipartites})
   \end{tabular} \\ \midrule
          $G=FQ_n$ &$2$~\cite{MANUEL2012} &
          $n2^n$ (Theorem~\ref{thm:folded-hypercubes}) \\ \midrule
          $G=W_{2^n}$  &
                         \begin{tabular}{c}
                           $n$\\
                           (Theorem~\ref{thm:maximal_dilation})
                         \end{tabular}
 &
   \begin{tabular}{r}
     $(n+2)2^{n-1}$\\
     (Theorem~\ref{thm:wheel_wirelength})
   \end{tabular}
  \\ \midrule
  \begin{tabular}{r}
    $G=K_{2^{n/2}}\cp K_{2^{n/2}}$\\ ($n$ even)
  \end{tabular}
  & $\leq n/2$
 &
   \begin{tabular}{c}
     $n2^{3n/2-2}$\\
     (Theorem~\ref{thm:cartesian_prod_wirelength})
   \end{tabular}
  \\ \bottomrule
\end{tabular}
\end{center}
\end {table}

\section{Preliminaries}\label{sec:prelim}

In this section, we give basic definitions and preliminaries related to embedding problems. Let $G$
and $H$ be finite graphs. Then an \emph{embedding of $G$ into $H$} is a pair $(f,P_f)$, where
$f: V(G)\to V(H)$ is a one-to-one mapping and $P_f$ is a mapping from $E(G)$ to the set of paths in
$H$ such that each edge $e=uv\in E(G)$ is assigned a path $P_f(e)$ in $H$ between $f(u)$ and
$f(v)$. By abuse of language we will also refer to an embedding $(f, P_f)$ simply by $f$.

The length of $P_{f}(e)$ is the \emph{dilation} $\dil_f(e)$ of the edge $e$ (with respect to the
embedding $(f,P_f)$). The maximum dilation over all edges of $G$ is called \emph{the dilation of the
  embedding}. The \emph{dilation of embedding $G$ into $H$} is the minimum dilation taken over all
embeddings of $G$ into $H$ and denoted by $\dil(G,H)$, cf.~\cite{Bezrukov1998}. In other words,
\[\dil(G,H)=\min_f\max_{e\in E(G)} \dil_f(e)\,,\]
where the minimum is taken over all injections $f:V(G)\to V(H)$. It is not necessary to specify
$P_f$ when we are interested in the dilation of embedding $G$ into $H$, because we can take
$P_f(uv)$ to be any shortest path between $f(u)$ and $f(v)$. The \emph{diameter} of $H$, denoted by
$\diam(H)$, is the maximum distance between any two vertices of $H$, and this provides a trivial
upper bound for $\dil(G,H)$.

The wirelength of an embedding $f= (f, P_f)$ of $G$ into $H$ is
\begin{equation}\label{eq:WL}
  WL_{f}(G,H) = \sum_{e_G\in E(G)}\dil_f(e_G)\,,
\end{equation}
and the \emph{wirelength of embedding $G$ into $H$} is
\[WL(G,H)=\underset{f}\min ~WL_{f}(G,H)\,,\]
see~\cite{Manuel2009}. We also refer
to~\cite{Bezrukov2000, Bezrukov1998, Opatrny2000} for the search of embeddings of $G$ into $H$ that
attain the value $WL(G,H)$.

Let $f = (f, P_f)$ be an embedding of $G$ into $H$. For $e=uv\in E(H)$, let $EC_{f}(e)$ denote
the number of edges $e'\in E(G)$ such that $e$ is in the path $P_{f}(e')$ between $f(u)$ and
$f(v)$ in $H$. Then the \emph{edge congestion of $f$} is
\[EC_{f}(G,H)=\max_{e\in E(H)} EC_{f}(e)\]
and the \emph{edge congestion of embedding $G$ into $H$} is
\[EC(G,H)=\min_fEC_{f}(G,H)\,,\]
where the minimum is taken over all embeddings $f = (f, P_f)$ of $G$ into $H$. Note that
\begin{equation}  \label{eq:WL_by_EC}
  WL_{f}(G,H) = \sum_{e_H\in E(H)}EC_{f}(e_H)\,.
\end{equation}
The above problems can be specialized to several distance problems. For example, if the host graph
is a path, then the dilation of embedding graph into a path is the \emph{bandwidth problem}. Further
distance problems are listed in Table~\ref{table2}, where $P_n$, $C_n$ and $K_n$, are the path, the
cycle, and the complete graph of order $n$, respectively, and $K_{a,b}$ is the complete bipartite
graphs on $a+b$ vertices.

\begin {table}[htb]
\caption {Various embedding problems}\label{table2}
\begin{center}
\begin{tabular}{ScScScSc}\toprule
\vspace{0.1 cm}
\textbf{\begin{tabular}{c}
          Embedding \\
          $f:G \rightarrow H$ \\
        \end{tabular}
  } & \textbf{Dilation} & \textbf{Wirelength}  & \textbf{Congestion} \\ \midrule
  \textbf{$H=P_n$} & Bandwidth~\cite{Harper1966} & MinLA~\cite{Adolphson1973} & Cutwidth~\cite{Harper1964}\\
  \textbf{$H=C_n$} & Cyclic bandwidth~\cite{Leung1984} & Cyclic wirelength~\cite{Bezrukov1998a} & Cyclic cutwidth~\cite{Chavez1998} \\
  \textbf{$G=K_n$} & ${\rm diam}(H)$~\cite{Wu1985} & Wiener index of $H$~\cite{Wiener1947} & Forwarding index of $H$~\cite{Heydemann1989} \\
  \textbf{$G=K_{a,b}$} & Makespan~\cite{SKORIN-KAPOV1993}  & Average load & Maximum load \\ \bottomrule
\end{tabular}
\end{center}
\end {table}

Let $n\geq 1$. The \emph{$n$-cube} ($n$-dimensional hypercube) has vertex set $\{0,1\}^n$, and
vertices $x,y\in V(Q_{n})$ are adjacent if and only if the corresponding binary strings differ
exactly in one bit. Equivalently, the vertices of $Q_{n}$ can be identified with integers
$0,1,\dots,2^n-1$, where vertices $i$ and $j$ are adjacent if and only if $i-j=\pm 2^{p}$ for some
integer $p\geq 0$.

For disjoint subsets $A,B\subseteq V(G)$ we will use the notation
$E_G(A,B)=\{uv\in E(G)~|~ u\in A,\, v\in B\}$ and $E_G(A)=\{uv~|~ u,v\in A\}$. If $x$ is a vertex of
a graph $G$, then $N_i(x)$ is the set of vertices in $G$ at distance $i$ from $x$. The maximum
degree of $G$ is denoted by $\Delta(G)$ and its order by $n(G)$.

\section{Graphs with large dilation of embedding into hypercubes}\label{sec:dilation}

The dilation, the wirelength, and the edge congestion problem are different in the sense that an
embedding that gives the minimum dilation need not give the minimum congestion (wirelength) and
vice-versa. From $\diam(Q_n)=n$ we obtain immediately the upper bound $\dil(G,Q_n)\leq n$ for every
graph $G$. To characterize graphs $G$ of order $2^n$ with $\dil(G, Q_n)=n$, we need the following
concept. A pair $(u,u')\in V(G)\times V(G)$, is said to be an antipodal pair (of
vertices) if the distance between $u$ and $u'$ is equal to the diameter of
$G$. If this is the case, we also say that $u'$ is an antipodal vertex of $u$
  and $u$ is an antipodal vertex of $u'$. In the hypercube $Q_n$, every vertex $u\in V(Q_n)$ has a
unique antipodal vertex which we denote by $\bar u$. If $G$ is a graph, then the set of pairs of
vertices $X = \{\{x_1,y_1\}, \dots, \{x_k,y_k\}\}$ forms an \emph{anti-matching}, if
$x_1y_1\notin E(G),\ldots, x_ky_k\notin E(G)$. That is, $X$ is an anti-matching if
$x_1y_1, \ldots, x_ky_k$ form a matching in the \emph{complement of $G$}, that is, the
  graph which has the same vertex set as $G$, and edge set
  $\{xy\,\vert\,x,y\in V(G),\,xy\not\in E(G)\}$. In addition, we say that $X$ is a {\em perfect
  anti-matching} if $2|X| = n(G)$. Now we have the following characterization.

\begin{theorem}\label{thm:maximal_dilation}
  Let $G$ be a graph of order $2^n$. Then $\dil(G,Q_n)\leq n-1$ if and only if $G$ contains a
  perfect anti-matching.
\end{theorem}
\begin{proof}
  Let $\dil(G,Q_n)\leq n-1$ and let $(f, P_f)$ be a corresponding embedding of $G$ into $Q_n$, that
  is, $\dil_f(G,Q_n)=\dil(G,Q_n)$. Consider the partition
  $X' = \{ \{u,\overline{u}\}~|~u\in V(Q_n)\}$ of $V(Q_n)$ into $2^{n-1}$ antipodal
  pairs. If $e=\{f^{-1}(u),f^{-1}(\bar u)\}\in E(G)$ for some $u\in V(Q_n)$, then
  \[\dil_f(G,Q_n)\geq\dil_f(e)\geq\dist_{Q_n}(u,\bar u)=n,\]
  and this contradicts the assumption. We conclude that $f^{-1}(u)$ is not adjacent to
  $f^{-1}(\overline{u})$, and it follows that
  $X = \{ \{f^{-1}(u),f^{-1}(\overline{u})\}~|~u\in V(Q_n)\}$ is a perfect anti-matching.

  Conversely, let $X$ be a perfect anti-matching of $G$. We now define an embedding $(f, P_f)$ of
  $G$ into $Q_n$ by mapping each pair $\{x,y\}$ from $X$ into an antipodal pair of vertices of
  $Q_n$, that is, $\{f(x), f(y)\} = \{v, \overline{v}\}$ for some $v\in V(Q_n)$. Since each vertex
  of $Q_n$ has a unique antipodal vertex, it follows that $\dil_f(e) \le n-1$ for each edge $e$ of
  $G$ and hence $\dil(G, Q_n)\leq n-1$.
\end{proof}
To determine a large class of graphs $G$ with $\dil(G,Q_n)=n-1$ we state the following lemma that
could be applied elsewhere.
\begin{lemma}\label{lem:dilation lower bound}
  If $G$ is a graph of order $2^n$, then
  \[\dil(G,Q_n)\geq\max\left\{k\,:\,\sum_{i=1}^{k-1}\binom{n}{i}<\Delta(G)\right\}.\]
\end{lemma}
\begin{proof}
  Let $k$ be a positive integer satisfying $\binom{n}{0}+\dots+\binom{n}{k-1}<\Delta(G)$, and
  let $(f, P_f)$ be an embedding of $G$ into $Q_n$. Let $u$ be a vertex of $G$ with
  $\deg(u) = \Delta(G)$ and let $f(u) = v\in V(Q_n)$. Setting
  $X = N_1(v)\cup\dots\cup N_{k-1}(v)$, the assumption on $k$ implies that
  \[|X| = |N_1(v)|+\cdots+|N_{k-1}(v)| = \sum_{i=1}^{k-1} \binom{n}{i} <
    \Delta(G)\,.\] Thus, there exists at least one vertex $w$ of $G$ adjacent to $u$ such that
  $f(w)\in V(Q_n)\setminus X$. It follows that the length of $P_f(uw)$ is at least $k$ and
  consequently $\dil(G, Q_n)\geq k$.
\end{proof}
Combining Lemma~\ref{lem:dilation lower bound} with Theorem~\ref{thm:maximal_dilation} we obtain the following
corollary.
\begin{corollary}\label{cor:dilation_n-1}
  If $G$ is a graph of order $2^n$ with $\Delta(G)\ge 2^n-n-1$ that contains a perfect
  anti-matching, then $\dil(G,Q_n) = n-1$.
\end{corollary}
\begin{proof}
  By Theorem~\ref{thm:maximal_dilation}, $\dil(G,Q_n) \le n-1$ and the reverse inequality follows
  from Lemma~\ref{lem:dilation lower bound} and
  \[\sum_{i=1}^{n-2}\binom{n}{i}=2^n-\binom{n}{0}-\binom{n}{n-1}-\binom{n}{n}=2^n-n-2<\Delta(G).\qedhere\]
\end{proof}
For illustration, consider $G=K_{4,4,4,4}$. Then $\Delta(G) = 12 \ge 2^4-4-1$. Since $G$ contains a
perfect anti-matching, $\dil(G, Q_4) = 3$. In Figure~\ref{fig1} a corresponding embedding is shown.
\begin{figure}[htb]
  \centering
  \begin{tikzpicture}[every node/.style={circle,fill=black,outer sep=1pt,inner sep=1pt}]
    \node[label=225:{$0$}] (v0) at (0,1.5) {};
    \node[label=225:{$1$}] (v1) at (0.5,1) {};
    \node[label=225:{$14$}] (v14) at (1,0.5) {};
    \node[label=225:{$15$}] (v15) at (1.5,0) {};
    \node[label=-45:{$2$}] (v2) at (2.5,0) {};
    \node[label=-45:{$3$}] (v3) at (3,.5) {};
    \node[label=-45:{$12$}] (v12) at (3.5,1) {};
    \node[label=-45:{$13$}] (v13) at (4,1.5) {};
    \node[label=45:{$4$}] (v4) at (4,2.5) {};
    \node[label=45:{$5$}] (v5) at (3.5,3) {};
    \node[label=45:{$10$}] (v10) at (3,3.5) {};
    \node[label=45:{$11$}] (v11) at (2.5,4) {};
    \node[label=135:{$6$}] (v6) at (1.5,4) {};
    \node[label=135:{$7$}] (v7) at (1,3.5) {};
    \node[label=135:{$8$}] (v8) at (.5,3) {};
    \node[label=135:{$9$}] (v9) at (0,2.5) {};
    \draw (v0) -- (v2); \draw (v0) -- (v3); \draw (v0) -- (v12); \draw (v0) -- (v13);
    \draw (v0) -- (v4); \draw (v0) -- (v5); \draw (v0) -- (v10); \draw (v0) -- (v11);
    \draw (v0) -- (v6); \draw (v0) -- (v7); \draw (v0) -- (v8); \draw (v0) -- (v9);
    \draw (v1) -- (v2); \draw (v1) -- (v3); \draw (v1) -- (v12); \draw (v1) -- (v13);
    \draw (v1) -- (v4); \draw (v1) -- (v5); \draw (v1) -- (v10); \draw (v1) -- (v11);    
    \draw (v1) -- (v6); \draw (v1) -- (v7); \draw (v1) -- (v8);  \draw (v1) -- (v9);
    \draw (v14) -- (v2); \draw (v14) -- (v3); \draw (v14) -- (v12); \draw (v14) -- (v13);
    \draw (v14) -- (v4); \draw (v14) -- (v5); \draw (v14) -- (v10); \draw (v14) -- (v11);
    \draw (v14) -- (v6); \draw (v14) -- (v7); \draw (v14) -- (v8); \draw (v14) -- (v9);
    \draw (v15) -- (v2); \draw (v15) -- (v3); \draw (v15) -- (v12); \draw (v15) -- (v13);
    \draw (v15) -- (v4); \draw (v15) -- (v5); \draw (v15) -- (v10); \draw (v15) -- (v11);
    \draw (v15) -- (v6); \draw (v15) -- (v7); \draw (v15) -- (v8); \draw (v15) -- (v9);    
    \draw (v2) -- (v4); \draw (v2) -- (v5); \draw (v2) -- (v10); \draw (v2) -- (v11);
    \draw (v2) -- (v6); \draw (v2) -- (v7); \draw (v2) -- (v8); \draw (v2) -- (v9);
    \draw (v3) -- (v4); \draw (v3) -- (v5); \draw (v3) -- (v10); \draw (v3) -- (v11);
    \draw (v3) -- (v6); \draw (v3) -- (v7); \draw (v3) -- (v8); \draw (v3) -- (v9);
    \draw (v12) -- (v4); \draw (v12) -- (v5); \draw (v12) -- (v10); \draw (v12) -- (v11);
    \draw (v12) -- (v6); \draw (v12) -- (v7); \draw (v12) -- (v8); \draw (v12) -- (v9);
    \draw (v13) -- (v4); \draw (v13) -- (v5); \draw (v13) -- (v10); \draw (v13) -- (v11);
    \draw (v13) -- (v6); \draw (v13) -- (v7); \draw (v13) -- (v8); \draw (v13) -- (v9);
    \draw (v4) -- (v6); \draw (v4) -- (v7); \draw (v4) -- (v8); \draw (v4) -- (v9);
    \draw (v5) -- (v6); \draw (v5) -- (v7); \draw (v5) -- (v8); \draw (v5) -- (v9);
    \draw (v10) -- (v6); \draw (v10) -- (v7); \draw (v10) -- (v8); \draw (v10) -- (v9);
    \draw (v11) -- (v6); \draw (v11) -- (v7); \draw (v11) -- (v8); \draw (v11) -- (v9);
    \draw[ultra thick,-stealth] (5,2) to node[above,fill=none] {$f$} (6.5,2);
    \node[label=225:{$6$}] (w7) at (7,1) {};
    \node[label=225:{$7$}] (w8) at (9,1) {};
    \node[label=135:{$4$}] (w5) at (7.5,1.5) {};
    \node[label=135:{$5$}] (w6) at (9.5,1.5) {};
    \node[label=225:{$2$}] (w3) at (7,3) {};
    \node[label=225:{$3$}] (w4) at (9,3) {};
    \node[label=135:{$0$}] (w1) at (7.5,3.5) {};
    \node[label=135:{$1$}] (w2) at (9.5,3.5) {};
    \node[label=225:{$14$}] (w15) at (11,1) {};
    \node[label=225:{$15$}] (w16) at (13,1) {};
    \node[label=45:{$12$}] (w13) at (11.5,1.5) {};
    \node[label=45:{$13$}] (w14) at (13.5,1.5) {};
    \node[label=-135:{$10$}] (w11) at (11,3) {};
    \node[label=-135:{$11$}] (w12) at (13,3) {};
    \node[label=45:{$8$}] (w9) at (11.5,3.5) {};
    \node[label=45:{$9$}] (w10) at (13.5,3.5) {};
    \draw[thick] (w7) -- (w8) -- (w6) -- (w5) -- (w1) -- (w2) -- (w4) -- (w3) -- (w7) -- (w5);
    \draw[thick] (w15) -- (w16) -- (w14) -- (w13) -- (w9) -- (w10) -- (w12) -- (w11) -- (w15) -- (w13);
    \draw[thick] (w1) -- (w3);
    \draw[thick] (w2) -- (w6);
    \draw[thick] (w4) -- (w8);
    \draw[thick] (w9) -- (w11);
    \draw[thick] (w10) -- (w14);
    \draw[thick] (w12) -- (w16);
    \draw[thick,bend left=15] (w7) to (w15);
    \draw[thick,bend left=15] (w8) to (w16);
    \draw[thick,bend left=25] (w5) to (w13);
    \draw[thick,bend left=25] (w6) to (w14);
    \draw[thick,bend left=15] (w3) to (w11);
    \draw[thick,bend left=15] (w4) to (w12);
    \draw[thick,bend left=25] (w1) to (w9);
    \draw[thick,bend left=25] (w2) to (w10);
    \node[fill=none,rectangle] at (2,-.8) {$K_{4,4,4,4}$};
    \node[fill=none,rectangle] at (10,-.8) {$Q_4$};
  \end{tikzpicture}
\caption{An embedding of $K_{4,4,4,4}$ into $Q_4$ with dilation $3$. The vertices of $K_{4,4,4,4}$ are labeled such
  that we can take $f$ to be the identity. The perfect anti-matching in $G$ consists of the pairs
  $\{i,15-i\}$ for $i=0,1,\dots,7$.}\label{fig1}
\end{figure}

\section{Two lower bounds in terms of the bisection width}\label{sec:two-bounds}

In this section, we give a general lower bound on the edge congestion and on the wirelength of
embedding into hypercubes. Both bound involve the bisection width that is defined as follows. The
\emph{bisection width} $BW(G)$ of a graph $G$ is the minimum number of edges necessary in an edge cut
so that the sizes of the two sides of the cut differ by at most one.

The first lower bound is established by the following averaging argument.
\begin{theorem}\label{congestion lower bound general}
  If $G$ and $H$ are graphs of the same order, then
  \[EC(G, H)\geq \frac{BW(G)}{BW(H)}\,.\]
\end{theorem}
\begin{proof}
  Set $n = n(G)$ $(= n(H))$. Let $A\cup B$ be a partition of $V(H)$ such that
  $|A|=\lceil n/2\rceil$, $|B|=\lfloor n/2\rfloor$, and $|E_H(A,B)|=BW(H)$, and let
  $f:G\to H$ be an embedding with $EC_f(G,H)=EC(G,H)$. Then
  \begin{equation}\label{eq:bw_bound}
    BW(G)\leq \left\lvert E_G(f^{-1}(A), f^{-1}(B))\right\rvert\leq \sum_{e\in E_H(A,B)} EC_f(e)\,.
  \end{equation}
 where the first inequality follows from the definition of the bisection width, and the second one
 from the fact that for every edge $uv\in E_G(f^{-1}(A), f^{-1}(B))$ the path $P_f(uv)$ has to use
 at least one edge from $E_H(A,B)$. Finally, we conclude
 \begin{multline*}
   \frac{BW(G)}{BW(H)}=\frac{BW(G)}{\left\lvert E_H(A,B)\right\rvert}\leq\frac{1}{\left\lvert
       E_H(A,B)\right\rvert}\sum_{e\in E_H(A,B)} EC_f(e) \leq\max_{e\in E_H(A,B)} EC_f(e)\\
   \leq\max_{e\in E(H)} EC_f(e)=EC_f(G,H)=EC(G,H).  \qedhere
 \end{multline*}
\end{proof}

\begin{corollary}
\label{congestion lower bound}
If $G$ is a graph of order $2^n$, then
\[EC(G, Q_n)\geq \frac{BW(G)}{2^{n-1}}\,.\]
\end{corollary}
\begin{proof}
  The result follows from Theorem~\ref{congestion lower bound general} and the fact that
  $BW(Q_n)=2^{n-1}$, see~\cite{Stacho1998}.
\end{proof}

The lower bound for the wirelength is the following.
\begin{theorem}\label{wirelength lower bound}
If $G$ is a graph of order $2^n$, then $WL(G, Q_n)\geq n\cdot BW(G)\,.$
\end{theorem}
\begin{proof}
  Let $E(Q_n)=S_1\cup\dots\cup S_n$ be the partition of $E(Q_n)$ where $S_i$ consists of the edges
  whose incident vertices differ in the $i$th coordinate. Then $Q_n\setminus S_i$ consists of two
  vertex-disjoint $(n-1)$-cubes. Moreover, $ |S_i|=2^{n-1} = BW(Q_n)$, and we can now estimate as
  follows, where the minimum is taken over all embeddings of $G$ into $Q_n$:
  \[ WL(G,Q_n) = \min\limits_f \sum_{e\in E(Q_n)} EC_{f}(e) = \min\limits_f \sum_{i=1}^n \sum_{e\in
      S_i} EC_{f}(e) \stackrel{\eqref{eq:bw_bound}}{\geq} \min\limits_f \sum_{i=1}^n BW(G) = n\cdot BW(G)\,.\]
  The first equality above holds by~\eqref{eq:WL} and~(\ref{eq:WL_by_EC}), and the second equality follows because the sets
  $S_i$ form a partition of $E(Q_n)$.
\end{proof}

\section{Embeddings of complete multipartite graphs}\label{sec:complete-multipartite}

In this section, we consider the complete $p$-partite graph $G = K_{n_1,\dots, n_p}$ of order
$2^n$. Recall that $G$ contains $p$ independent sets containing $n_i$, $i\in [p]$, vertices, and
all possible edges between vertices from different parts.

From Theorem~\ref{thm:maximal_dilation} it follows that $\dil(G,Q_n) = n$ if at least one of the
$p_i$ is odd and $\dil(G,Q_n)\leq n-1$ otherwise (that is, if all the $p_i$ are even). We will next determine the wirelengths of embedding certain complete $p$-partite graphs into $n$-cubes, and start by determining the bisection widths of balanced complete multipartite graphs.
\begin{lemma}\label{lem:bw_multipartite}
  If $G$ is the complete $t$-partite graph $K_{2r,\dots,2r}$, with $2tr$ vertices, $t\geq 2$, $r\geq
  1$, then
  \[BW(G)=r^2t(t-1)\,.\]
\end{lemma}
\begin{proof}
  Let $V(G)=A\cup B$ be a partition with $|A|=|B|=tr$ such that $A$ and $B$ each contains half
  of the vertices in every maximal independent set. Then
  \[BW(G)\leq\left\lvert E_G(A,B)\right\rvert =\lvert E(G)\rvert-\lvert E_G(A)\rvert-\lvert
    E_G(B)\rvert=4r^2\binom{t}{2}-2r^2\binom{t}{2}=r^2t(t-1).\]
  For the reverse inequality, let $V(G)=A\cup B$ be a partition with $|A|=|B|=tr$ and $\lvert
  E_G(A,B)\rvert=BW(G)$. For $i=1,\dots,t$, let $n_i$ be the number of vertices of $A$ in the $i$-th
  independent set. Then we apply the Cauchy-Schwarz inequality to obtain the bound
  \[\lvert E_G(A)\rvert=\sum_{1\leq i<j\leq
      t}n_in_j=\frac12\left(\left(\sum_{i=1}^tn_i\right)^2-\sum_{i=1}^tn_i^2\right)
    =\frac12\left((tr)^2-\sum_{i=1}^tn_i^2\right)\leq\frac{r^2t(t-1)}{2}\,.\]
  For the same reason, $\lvert E_G(B)\rvert\leq r^2t(t-1)/2$, and consequently,
 \[BW(G)=\left\lvert E_G(A,B)\right\rvert =\lvert E(G)\rvert-\left(\lvert E_G(A)\rvert+\lvert
      E_G(B)\rvert\right)
    \geq 4r^2\binom{t}{2}-r^2t(t-1)\\
    =r^2t(t-1)\,.\qedhere\]
\end{proof}
We now want to determine the wirelength for embedding the complete $2^p$-partite graph
$G=K_{2^{n-p},\ldots, 2^{n-p}}$, $1\leq p<n$ into $Q_n$. In the proof of the following theorem, it
will be convenient to have an explicit notation for the function which maps a non-negative integer to
the bitstring corresponding to its binary representation. More specifically, we will use the function $\phi:\nats\to
\{0,1\}^{\nats}$, defined by
\[\phi\left(2^{k-1}x_1+2^{k-2}x_2+2^{k-3}x_3+\dots+2^1x_{k-1}+2^0x_k\right)=x_1x_2\dots x_k.\]
\begin{theorem}\label{thm:embedding_multipartites}
  If $G$ is the complete $2^p$-partite graph $K_{2^{n-p},\ldots, 2^{n-p}}$, where $1 \leq p < n$, then
\[ WL(G,Q_n) = n\cdot 2^{2n-p-2}~(2^{p}-1)\,.\]
\end{theorem}
\begin{proof}
  Combining Lemma~\ref{lem:bw_multipartite} for $t=2^p$ and $r=2^{n-p-1}$ with Theorem~\ref{wirelength lower bound}, we obtain
  \[ WL(G,Q_n) \geq n\cdot BW(G)=n\cdot 2^{2n-p-2}~(2^{p}-1).\]
  In order to prove the reverse inequality we construct an embedding $f:G\to Q_n$ in the following way. Let
  $V(G)=V_1\cup\dots\cup V_{2^p}$ be the partition of $V(G)$ into $2^p$ independent sets of size
  $2^{n-p}$. Label the vertices of $V(G)$ with the numbers $0,1,2,\dots,2^n-1$, in such a way that
  \begin{equation}\label{eq:V_i}
    V_i =\left\{j2^{p+1}+i-1\,:\,j=0,1,2,\dots,2^{n-p-1}-1\right\}\cup\left\{j2^{p+1}-i\,:\,j=1,2,3,\dots,2^{n-p-1}\right\}.
  \end{equation}
  Recall that $V(Q_n) = \{0,1\}^n$ and define $f:V(G)\to V(Q_n)$ by mapping $x\in V(G)$ to its
  binary representation. In other words, $f$ is the restriction of $\phi$ to the set
  $\{0,1,\dots,2^n-1\}$, that is, $x=2^{n-1}x_1+2^{n-2}x_2+\dots+2^0x_n$ is mapped to
  $f(x) = \phi(x)= x_1 x_2 \dots x_n$. We will now show that for this embedding,
  $WL_f(G,Q_n)=n\cdot 2^{2n-p-2}~(2^{p}-1)$, which concludes the proof. For $x,y\in\{0,1\}^n$, let
  $d(x,y)$ denote their Hamming distance, that is,
  $d(x,y)=\lvert x_1-y_1\rvert+\dots+\lvert x_n-y_n\rvert$, which is equal to the distance between
  $x$ and $y$ in $Q_n$. It follows from~(\ref{eq:V_i}) that, for every $i\in\{1,\dots,2^p\}$,
    \[f(V_i)=\{0,1\}^{n-p-1}\times\{\phi(i-1),\,\phi(2^{p+1}-i)\}.\]
    For instance, for $p=3$ and $n=p+3=6$, 
     \begin{align*}
       f(V_1) &= \left\{
                \begin{matrix}
                  000000,\\ 010000,\\100000,\\ 110000,\\001111,\\011111,\\101111,\\111111\phantom{,}
                \end{matrix}
 \right\}, & f(V_2) &= \left\{
                \begin{matrix}
                  000001,\\ 010001,\\100001,\\ 110001,\\001110,\\011110,\\101110,\\111110\phantom{,}
                \end{matrix}
 \right\}, & f(V_3) &= \left\{
                \begin{matrix}
                  000001,\\ 010010,\\100010,\\ 110010,\\001101,\\011101,\\101101,\\111101\phantom{,}
                \end{matrix}
 \right\}, & \text{etc.}
     \end{align*}
     If $x$ is an arbitrary vertex in $V_i$, and $k\in\{1,\dots,n\}$ is any of the positions, then exactly half of the
     elements of $f(V_i)$ differ from $f(x)$ in position $k$. For $k\in\{1,2,\dots,n-p-1\}$ this is
     because the first $n-p-1$ positions run twice through the vertices of $Q_{n-p-1}$, and for
     $k\in\{n-p,\dots,n\}$ it follows from the fact that $\phi(i-1)$ and $\phi(2^{p+1}-i)$ are
     antipodal vertices in $Q_{p+1}$. As a consequence, for any fixed $x\in V_i$,
     \[\sum_{y\in V_i\setminus x}d(f(x),f(y))=\sum_{k=1}^{n}\frac{1}{2}\left\lvert
         V_i\right\rvert=n2^{n-p-1},\]
     hence
     \[\sum_{\{x,y\}\in\binom{V_i}{2}}d(f(x),f(y))=\frac12\sum_{x\in V_i}\sum_{y\in V_i\setminus
         x}d(f(x),f(y))=\frac12\sum_{x\in V_i}n2^{n-p-1}=n2^{2(n-1-p)}.\]
     As this is true for every $i\in\{1,2,\dots,2^p\}$,
\begin{multline*}
  WL_f(G,Q_n)=\sum_{\{x,y\}\in\binom{V(Q_n)}{2}}d(x,y)-\sum_{i=1}^{2^p}\sum_{\{x,y\}\in\binom{V_i}{2}}d(f(x),f(y))\\
  =n2^{2n-2}-2^pn2^{2(n-1-p)}=n2^{2n-2-p}(2^p-1)\,.\qedhere
\end{multline*}
\end{proof}

\section{Embeddings of folded hypercubes}\label{sec:folded hypercubes}

The $n$-dimensional \emph{folded hypercube} $FQ_{n}$ is the graph obtained from $Q_{n}$ by adding,
for every vertex $x=x_{1}\ldots x_{n}$, the edge between $x$ and its antipodal vertex
$\overline{x}=\overline{x}_{1}\ldots \overline{x}_{n}$, where $\overline{x}_i=1-x_i$,
cf.~\cite{Xu2013}. Folded hypercubes are of wide current interest in reliability and fault tolerance
of interconnection networks, see~\cite{Cheng2018,Kuo2017,Liu2017,Yang2017}. In this section, we add
to the known results on folded hypercubes their wirelength.

\begin{theorem}\label{thm:folded-hypercubes}
  For all $n\ge 2$, $WL(FQ_n,Q_n)= n2^n$.
\end{theorem}
\begin{proof}
  In~\cite{Rajasingh2011} it was proved that $BW(FQ_n)=2^{n}$, and with Theorem~\ref{wirelength
    lower bound} this implies $WL(FQ_n,Q_n)\geq n2^n$.  To prove the reverse inequality, let
  ${\rm id}:FQ_n\to Q_n$ be the identity embedding, that is, ${\rm id}(v) = v$, for every
  $v\in V(FQ_n)$ (and where the second $v$ is considered as a vertex of $Q_n$). For this embedding,
  the $n2^{n-1}$ edges of $FQ_n$ that are also the edges of $Q_n$ contribute $n2^{n-1}$ to the
  wirelength. and each of the $2^{n-1}$ antipodal edges contributes $n$ to the wirelength. We
  conclude
  \[WL_{{\rm id}}(FQ_n,Q_n) = n2^{n-1} + 2^{n-1} n = n\cdot 2^n\,.\qedhere\]
\end{proof}

\section{Embeddings of wheels}\label{sec:wheels}
Let $n\ge 4$. The {\em wheel} $W_n$ of order $n$ is the graph obtained from a cycle $C_{n-1}$ by
adding a new vertex $x$ and joining $x$ to all the vertices of the cycle. Vertex $x$ is called the
\emph{hub} of the wheel and the edges incident with $x$ are \emph{spokes} of the wheel. Since the
hub vertex is adjacent to all other vertices, the wheel does not contain an anti-matching, and by
Theorem~\ref{thm:maximal_dilation}, $\dil(W_{2^n}, Q_n) = n$. We next determine the wirelength
of embedding $W_{2^n}$ into $Q_n$.
\begin{theorem}\label{thm:wheel_wirelength}
If $n\ge 2$, then $WL(W_{2^n},Q_n)= (n+2)2^{n-1}$.
\end{theorem}
\begin{proof}
  For every embedding the spokes contribute
  \[\sum_{k=1}^nk\binom{n}{k}=n2^{n-1}\]
  to the wirelength. Since $Q_n$ is bipartite, the $2^n-1$ cycle edges contribute at least
  $2^n$. This gives a lower bound of $n2^{n-1}+2^n=(n+2)2^{n-1}$, and using a Gray code this bound
  can be achieved.
\end{proof}

\section{Embeddings of some Cartesian products}\label{sec:cartesian}
Recall that the {\em Cartesian product} $G\cp H$ of graphs $G$ and $H$ is the graph with the vertex
set $V(G) \times V(H)$, vertices $(g,h)$ and $(g',h')$ being adjacent if either $g=g'$ and
$hh'\in E(H)$, or $h=h'$ and $gg'\in E(G)$. For more information on this product see the
book~\cite{Imrich2008}. We will use the following result on the edge-isoperimetric problem for Cartesian products of complete graphs (see~\cite{Bezrukov1999} for a survey of related results).
\begin{theorem}\label{thm:lindsey}
{\rm (\cite{JohnH.Lindsey1964})}
  Let $G=K_{p_1}\cp K_{p_2}\cp\cdots\cp K_{p_t}$ with $p_1\leq p_2\leq\dots\leq p_t$, and let $m$ be
  an integer with $1\leq m\leq p_1p_2\cdots p_t$. If $H$ is a subgraph of $G$ with $\lvert
  V(H)\rvert=m$, then $\lvert E(H)\rvert\leq\lvert E(H^*)\rvert$, where $H^*$ is the subgraph
  induced by the first $m$ vertices in lexicographic order, where we take
  \[V(G)=\left\{(a_1,a_2,\dots,a_t)\,:\,1\leq a_i\leq p_i\text{ for }i=1,2\dots,t\right\},\]
  and the lexicographic order is defined by $(a_1,a_2,\dots,a_t)<(b_1,b_2,\dots,b_t)$ if there
  exists an index $i$ with $a_i<b_i$ and $a_j=b_j$ for all $j<i$.
\end{theorem}
As a corollary, we obtain the bisection widths for products of complete graphs.
\begin{corollary}\label{cor:bw_clique_products}
  If $G = K_{2p_1}\cp K_{p_2}\cp K_{p_3}\cp\cdots\cp K_{p_t}$ with $2p_1\leq p_2\leq
p_3\leq\dots\leq p_t$, then $BW(G)=p_1^2p_2\cdots p_t$.
\end{corollary}
\begin{proof}
  We have $\lvert E(G)\rvert=p_1p_2\cdots p_t\left(2p_1+p_2+p_3+\dots+p_t-t\right)$, and for any
  $A\subseteq V(G)$ with $\lvert A\rvert=p_1p_2\cdots p_t$, Theorem~\ref{thm:lindsey} implies
  \[\lvert E_G(A)\rvert\leq\frac12p_1p_2\cdots p_t\left(p_1+p_2+\dots+p_t-t\right)\]
  with equality if $A$ is the set of vertices $(a_1,\dots,a_t)$ with $a_1\leq p_1$ or the set of
  vertices $(a_1,\dots,a_t)$ with $a_1>p_1$. This implies that for every partition $V(G)=A\cup B$
  with $\lvert A\rvert=\lvert B\rvert=p_1p_2\cdots p_t$,
  \begin{multline*}
    \lvert E_G(A,B)\rvert=\lvert E(G)\rvert-\lvert E_G(A)\rvert-\lvert E_G(B)\rvert\\
    \geq p_1p_2\cdots p_t\left(2p_1+p_2+p_3+\dots+p_t-t\right)-p_1p_2\cdots
    p_t\left(p_1+p_2+\dots+p_t-t\right) =p_1^2p_2\cdots p_t
  \end{multline*}
  with equality if $A$ is the set of vertices $(a_1,\dots,a_t)$ with $a_1\leq p_1$.
\end{proof}

\begin{theorem}\label{thm:cartesian_prod_wirelength}
  Let $n$ be a positive even integer, and let $G=K_{2^{n/2}}\cp K_{2^{n/2}}$. Then
  \[WL(G,Q_{n})= n2^{3n/2-2}\,.\]
\end{theorem}
\begin{proof}
  From Corollary~\ref{cor:bw_clique_products} with $t=2$, $p_1=2^{n/2-1}$ and $p_2=2^{n/2}$, we obtain
  $BW(G)=2^{3n/2-2}$, and together with Theorem~\ref{wirelength lower bound} this implies the lower
  bound. To prove the upper bound we embed $G$ into $Q_n$, by mapping a vertex $(a,b)\in
  V(G)=\{0,1,\dots,2^{n/2}-1\}\times\{0,1,\dots,2^{n/2}-1\}$ with
  \begin{align*}
    a &= \sum_{i=1}^{n/2}a_i2^{n/2-i},& b &= \sum_{i=1}^{n/2}b_i2^{n/2-i}
  \end{align*}
  to the binary string $a_1a_2\ldots a_{n/2}b_1b_2\ldots b_{n/2}\in\{0,1\}^n=V(Q_n)$. For
  $i\in\{1,\dots,n\}$, let $V(Q_n)=A_i\cup B_i$ be the partition with
  $A_i=\{x\in V(Q_n)\,:\,x_i=0\}$ and $B_i=\{x\in V(Q_n)\,:\,x_i=1\}$. Then
  \[E(Q_n)=\bigcup_{i=1}^nE_{Q_n}(A_i,B_i)\]
  and since each of the sets $f^{-1}(A_i)$ and $f^{-1}(B_i)$ induces a subgraph of $G$ isomorphic to
  $K_{2^{n/2-1}}\cp K_{2^{n/2}}$, we obtain
  \begin{multline*}
    WL_f(G,Q_n)=\sum_{i=1}^n\sum_{e\in E_{Q_n}(A_i,B_i)}EC_f(e)=\sum_{i=1}^n\left\lvert
    E_G\left(f^{-1}(A_i),f^{-1}(B_i)\right)\right\rvert\\
    =\sum_{i=1}^nBW(G) = n\cdot BW(G)= n2^{3n/2-2}.\qedhere
  \end{multline*}
\end{proof}

\section*{Concluding Remarks}
In this paper, we have obtained lower bounds for dilation, wirelength, and edge congestion of an
embedding of graphs into $n$-cubes. In particular, we found lower bounds for congestion and
wirelength in terms the bisection width of the guest graph. This technique allows the computation of
the exact dilation and wirelength for a range of networks. An open problem remains to determine the
edge congestion of these networks.  It also opens up the study of embedding parameters which remains
unsolved for several good candidates of guest architectures. Another direction for extending our
work to make it more directly applicable is the study of weighted versions of the embedding problem.

\subsection*{Acknowledgments}
The work of R. Sundara Rajan was partially supported by Project No. 2/48(4)/
2016/NBHM-R\&D-II/11580, National Board of Higher Mathematics (NBHM), Department of Atomic Energy
(DAE), Government of India. Sandi Klav\v zar acknowledges the financial support from the Slovenian
Research Agency (research core funding No.\ P1-0297). Further, the authors would like to thank the
anonymous referees for their comments and suggestions. These comments and suggestions were very
helpful for improving the quality of this paper.

\bibliographystyle{amsplain}
\bibliography{embeddings}

\end{document}